\documentclass[11pt]{article}

\usepackage[T1]{fontenc}
\usepackage{lmodern}
\usepackage{amsmath,amssymb,amsthm,mathtools}
\usepackage[margin=1.08in]{geometry}
\usepackage{microtype}
\usepackage{tikz}
\usetikzlibrary{arrows.meta,positioning,calc}
\usepackage[section]{placeins}
\usepackage[hidelinks]{hyperref}

\newtheorem{theorem}{Theorem}[section]
\newtheorem{proposition}[theorem]{Proposition}
\newtheorem{lemma}[theorem]{Lemma}

\theoremstyle{definition}
\newtheorem{definition}[theorem]{Definition}
\theoremstyle{remark}
\newtheorem{remark}[theorem]{Remark}

\newcommand{\PP}{\mathbf P}
\newcommand{\Span}{\operatorname{span}}
\newcommand{\QQRT}{\mathcal Q}

\tikzset{
  modelcurve/.style={very thick,blue!65!black},
  construction/.style={thick,dashed,red!72!black},
  secondconstruction/.style={thick,dash dot,teal!70!black},
  guide/.style={thin,densely dotted,black!55},
  modelpoint/.style={circle,fill=black,inner sep=1.55pt},
  sumpoint/.style={circle,fill=blue!65!black,inner sep=1.8pt}
}

\title{Direct Group Laws on Classical Genus-One Models}
\author{Hongfeng Wu\\
College of Science, North China University of Technology\\
\texttt{whfmat@gmail.com}}
\date{}

\begin{document}

\maketitle

\begin{abstract}
This exposition constructs group laws by residual intersection on short Weierstrass cubics,
the curves $(u^2+u)(v^2+v)=d$, Edwards curves, Hessian curves, Jacobi
quartics, Jacobi intersections, Huff curves, and a symmetric QRT
biquadratic.  Lines give the operation on plane cubics, while planes give the
operation on smooth intersections of two quadrics in $\PP^3$.  Associativity
is proved by comparing two products of three line or plane equations.  After
the common zeros are cancelled, equality of the remaining zeros follows from
the elementary fact that a smooth projective curve of genus one admits no
nonconstant rational function with one simple pole.  The formulas for
$(u^2+u)(v^2+v)=d$ are valid in every characteristic and are written
explicitly in characteristics $2$ and $3$; the other models are treated in
characteristic different from $2$ and $3$.
\end{abstract}



\section{Introduction}

An explicit addition formula must be accompanied by proofs of closure,
identity, inverses, and associativity, together with a treatment of exceptional
inputs.  The constructions below are defined on the natural projective
completions of the displayed equations, so tangent and boundary cases are part
of the same operation.

Two constructions suffice.  On a smooth plane cubic with a chosen flex, the
usual chord--tangent rule is intrinsic to the displayed cubic.  On a smooth
elliptic normal quartic, that is, a smooth intersection of two quadrics in
$\PP^3$, a plane through three prescribed points has one residual intersection
point.  A model-specific involution turns that residual point into the sum.
Associativity in both cases is obtained by comparing two products of three
lines or three planes.  Their restrictions to the curve have all but one zero,
counted with multiplicity, in common.  The last zeros coincide because a
genus-one curve admits no rational function with one simple pole.  This is the
one-residual-point form of the Cayley--Bacharach principle used throughout the
paper. The exposition is intended for readers familiar with the basic language 
of smooth projective curves and intersection multiplicities.

Edwards introduced his normal form and addition law in \cite{Edwards2007}, and
Bernstein and Lange developed its arithmetic in \cite{BernsteinLange2007}.
Hales proved the Edwards group law computationally, and Hales and Raya later
formalized that proof \cite{Hales2016,HalesRaya2020}.  Hessian formulas appear
in the work discussed by Chudnovsky and Chudnovsky \cite{Chudnovsky1986} and
in \cite{JoyeQuisquater2001}.  References for Jacobi quartics and intersections
include \cite{LiardetSmart2001,BilletJoye2003}; the twisted intersection model
was developed by Feng, Nie, and Wu \cite{FengNieWu2013}.  Huff arithmetic is
treated in \cite{Huff1948,JoyeTibouchiVergnaud2010,WuFeng2012}.  Wu develops
the characteristic-uniform arithmetic of $C_d$ and a wider symmetric family
in \cite{Wu2026Cd}.  The QRT maps originate in the work of Quispel, Roberts,
and Thompson \cite{QRT1988,QRT1989}.  Zwegers's proof for a Weierstrass cubic
uses a parabola to obtain a symmetric three-point formula
\cite{Zwegers2024}; the relation with the Jacobi-quartic parabola is described
in Section~\ref{sec:Jacobi-quartic}.

Throughout, $k$ is a field.  Intersections are taken scheme-theoretically and
may be checked after extending scalars to an algebraic closure.  A residual
subscheme of length one obtained from $k$-rational data is Galois invariant and
therefore is a $k$-rational point.

The figures are schematic incidence diagrams.  Blue denotes the curve, red the
first chord or plane trace, and teal a second construction through the chosen
origin.  Labels and intersection multiplicities determine the construction;
Euclidean scale is not significant.

\section{Two residual-point principles}

\subsection{The one-pole lemma}

\begin{lemma}[One-pole lemma]\label{lem:one-pole}
Let $C$ be a smooth projective geometrically integral curve of genus one over
$k$.  A rational function on $C$ with at most one simple pole is constant.
\end{lemma}

\begin{proof}
It is enough to work over an algebraic closure of $k$.  A nonconstant rational
function $f$ defines a finite morphism $C\to\PP^1$, whose degree is the sum of
the pole orders of $f$.  If $f$ had at most one simple pole, this degree would
be one.  A finite morphism of degree one between smooth projective curves is an
isomorphism, contradicting $g(C)=1$ and $g(\PP^1)=0$.
\end{proof}

\begin{lemma}[Residual point]\label{lem:residual}
Let $C$ be as in Lemma~\ref{lem:one-pole}.  Suppose that two homogeneous forms
of the same degree restrict nontrivially to $C$, and that their zero schemes on
$C$ agree except possibly for one residual point of length one on each side.
Then the two residual points are equal.
\end{lemma}

\begin{proof}
The quotient of the two restricted forms is a rational function on $C$.  All
common zeros cancel with their local multiplicities.  If the residual points
were different, the quotient or its reciprocal would have exactly one simple
pole.  Lemma~\ref{lem:one-pole} rules this out.
\end{proof}

This formulation includes coincident points: equality is an equality of local
vanishing orders, not merely of sets.

\subsection{Plane cubics}

Let $C\subset\PP^2$ be a smooth plane cubic and let $O\in C(k)$ be a flex.
For $P\in C$, define $\iota(P)$ by
\[
             L(O,P)\cap C=O+P+\iota(P),
\]
where $L(O,O)$ is the tangent at $O$ and repeated points denote their
scheme-theoretic multiplicities.  The flex condition gives $\iota(O)=O$.
The same line contains $O,\iota(P),P$, so $\iota^2(P)=P$.

\begin{theorem}[Chord--tangent group law]\label{thm:cubic}
For $P,Q\in C(k)$, let the line through $P$ and $Q$ meet $C$ once more at $R$,
with tangent multiplicities when necessary, and define
\[
        P\oplus Q=\iota(R).
\]
Then $(C(k),\oplus)$ is an abelian group with identity $O$ and inverse
$-P=\iota(P)$.
\end{theorem}

\begin{proof}
The rule is symmetric in $P,Q$.  The line $OP$ has third point $\iota(P)$, so
$O\oplus P=P$.  The line through $P$ and $\iota(P)$ has residual point $O$;
hence $P\oplus\iota(P)=O$.

It remains to prove associativity.  Write
\[
 S=P\oplus Q,\qquad T=S\oplus R,
 \qquad V=Q\oplus R,\qquad T'=P\oplus V.
\]
Let $L(A,B)$ denote the line through $A,B$, and put $N(A)=L(O,A)$.
The two reducible cubics
\[
 L(P,Q)L(S,R)N(V),
 \qquad
 L(Q,R)L(P,V)N(S)
\]
cut on $C$, respectively,
\begin{align*}
 &O+P+Q+R+S+\iota(S)+V+\iota(V)+\iota(T),\\
 &O+P+Q+R+S+\iota(S)+V+\iota(V)+\iota(T').
\end{align*}
These sums denote effective zero-dimensional schemes.  Both divisors contain
the same effective divisor of degree eight, and each has degree nine.
Lemma~\ref{lem:residual} gives
$\iota(T)=\iota(T')$, hence $T=T'$.  Thus
$(P\oplus Q)\oplus R=P\oplus(Q\oplus R)$.
\end{proof}

\subsection{Intersections of two quadrics}

Let $C\subset\PP^3$ be a smooth complete intersection of two quadrics.  The
Koszul resolution
\[
 0\longrightarrow\mathcal O_{\PP^3}(-4)
 \longrightarrow\mathcal O_{\PP^3}(-2)^{\oplus2}
 \longrightarrow\mathcal O_{\PP^3}
 \longrightarrow\mathcal O_C\longrightarrow0
\]
remains exact after extension of the ground field and gives
$H^0(C_{\bar k},\mathcal O_{C_{\bar k}})=\bar k$.  Thus $C$ is geometrically
connected; since it is smooth, it is geometrically integral.  Its Hilbert
series $(1-t^2)^2/(1-t)^4$ gives degree four and arithmetic genus one, and
smoothness therefore gives geometric genus one.  Fix $O\in C(k)$.

A length-three subscheme spans a linear space of dimension at most two.  The
linear span of $O+P+Q$ is in fact a plane: if the subscheme were contained in
a line, the restriction of each defining quadric to that line would have at
least three zeros counted with multiplicity and hence would vanish
identically.  The line would then be contained in $C$, contrary to geometric
integrality.  Denote the resulting plane by $H(P,Q)$.  Thus
\[
        H(P,Q)\cap C=O+P+Q+R
\]
for a unique residual point $R$.  This definition includes tangent and
osculating multiplicities.  In particular, set
\[
       N(P)=H(O,P)=\Span(2O+P),
\]
so that $N(O)=\Span(3O)$ is the unique osculating plane determined by the
length-three subscheme $3O$.

Assume that $C$ has an involution $\eta:C\to C$ defined over $k$ such that
\begin{equation}\label{eq:normal-plane}
        N(P)\cap C=2O+P+\eta(P)
\end{equation}
for every $P$.  Put $J=\eta(O)$.

\begin{theorem}[Residual-plane group law]\label{thm:quartic}
With the preceding hypotheses, define
\begin{equation}\label{eq:plane-law}
        P\boxplus Q=\eta(R),
        \qquad H(P,Q)\cap C=O+P+Q+R.
\end{equation}
Then $(C(k),\boxplus)$ is an abelian group.  Its identity is $O$.  The inverse
of $P$ is the residual point $P^{-}$ in
\begin{equation}\label{eq:inverse-plane}
        \Span(O+J+P)\cap C=O+J+P+P^{-}.
\end{equation}
\end{theorem}

\begin{proof}
The definition is symmetric in $P,Q$.  Since $H(O,P)=N(P)$, its residual point
is $\eta(P)$, and therefore $O\boxplus P=P$.  In
\eqref{eq:inverse-plane} the residual point for the pair $P,P^{-}$ is $J$;
thus $P\boxplus P^{-}=\eta(J)=O$.

For associativity, set
\[
 S=P\boxplus Q,\quad T=S\boxplus R,
 \qquad V=Q\boxplus R,\quad T'=P\boxplus V.
\]
The products of planes
\[
 H(P,Q)H(S,R)N(V),
 \qquad
 H(Q,R)H(P,V)N(S)
\]
restrict to cubic forms on $C$.  Their zero schemes are
\begin{align*}
 &4O+P+Q+R+S+\eta(S)+V+\eta(V)+\eta(T),\\
 &4O+P+Q+R+S+\eta(S)+V+\eta(V)+\eta(T').
\end{align*}
Both divisors contain the same effective divisor of degree eleven; each has
degree twelve.  By
Lemma~\ref{lem:residual}, $\eta(T)=\eta(T')$.  Since $\eta$ is an involution,
$T=T'$, proving associativity.
\end{proof}

\begin{figure}[t]
\centering
\begin{tikzpicture}[
  node distance=8mm and 11mm,
  box/.style={draw,rounded corners,align=center,inner sep=5pt},
  arr/.style={-{Latex[length=2mm]},thick}
]
\node[box] (pq) {$P,Q$};
\node[box,right=of pq] (plane) {$H(P,Q)\cap C$\\$=O+P+Q+R$};
\node[box,right=of plane] (sum) {$P\boxplus Q=\eta(R)$};
\draw[arr] (pq)--(plane);
\draw[arr] (plane)--(sum);
\node[box,below=12mm of plane] (left)
 {$H(P,Q)H(S,R)N(V)$};
\node[box,right=18mm of left] (right)
 {$H(Q,R)H(P,V)N(S)$};
\node[below=7mm of left,align=center] (eleven)
 {eleven common zeros on $C$};
\draw[arr] (left)--(eleven);
\draw[arr] (right)--(eleven);
\node[box,below=7mm of eleven] (last)
 {$\eta((P\boxplus Q)\boxplus R)$\\
  $=\eta(P\boxplus(Q\boxplus R))$};
\draw[arr] (eleven)--(last);
\end{tikzpicture}
\caption{The plane rule and the residual-point proof of associativity.}
\label{fig:plane-law}
\end{figure}
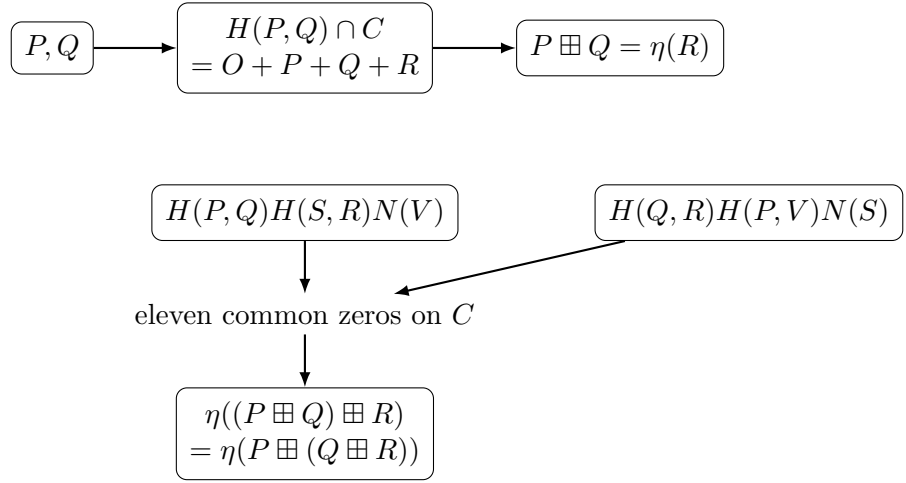

The following sections identify the relevant involution and verify
\eqref{eq:normal-plane} in the coordinates of each model.

\section{Short Weierstrass equations}\label{sec:short-Weierstrass}

Assume $\operatorname{char}k\ne2,3$.  Let
\begin{equation}\label{eq:short-Weierstrass}
 \mathcal W_{a,b}:\qquad
 Y^2Z=X^3+aXZ^2+bZ^3,\qquad 4a^3+27b^2\ne0,
\end{equation}
and take $O=[0:1:0]$.  An affine singular point would satisfy
$y=0$, $3x^2+a=0$, and $x^3+ax+b=0$; eliminating $x$ gives
$4a^3+27b^2=0$.  Conversely, vanishing of this discriminant produces a
common root of the cubic and its derivative over an algebraic closure.  The
point $O$ is nonsingular.  Thus the stated condition is precisely the
smoothness condition.  The line $Z=0$ meets the cubic at $O$ with
multiplicity three, so $O$ is a flex.  Moreover,
\[
       L(O,(x,y))\cap\mathcal W_{a,b}
       =O+(x,y)+(x,-y).
\]

\begin{definition}[Weierstrass chord law]
If the chord or tangent through $P,Q\in\mathcal W_{a,b}(k)$ has third
intersection $R$, define $P\oplus_W Q$ to be the reflection of $R$ in the
$x$-axis.
\end{definition}

\begin{theorem}\label{thm:Weierstrass-group}
The operation $\oplus_W$ makes $\mathcal W_{a,b}(k)$ an abelian group with
identity $O$ and inverse $-(x,y)=(x,-y)$.  If $x_1\ne x_2$, then
\begin{equation}\label{eq:Weierstrass-add}
 \lambda=\frac{y_2-y_1}{x_2-x_1},\qquad
 x_3=\lambda^2-x_1-x_2,\qquad
 y_3=-y_1+\lambda(x_1-x_3).
\end{equation}
For doubling, the same formulas hold with
\begin{equation}\label{eq:Weierstrass-double-slope}
             \lambda=\frac{3x_1^2+a}{2y_1}
\end{equation}
whenever the displayed denominator is nonzero.
\end{theorem}

\begin{proof}
The description of the line through $O$ shows that the involution in
Theorem~\ref{thm:cubic} is $(x,y)\mapsto(x,-y)$.  Hence that theorem proves
all group axioms.  For the affine formulas, write the chord as
$y=\lambda x+\nu$.  Substitution in \eqref{eq:short-Weierstrass} gives
\begin{equation}\label{eq:Weierstrass-elimination}
 (\lambda x+\nu)^2-(x^3+ax+b)
 =-\bigl(x^3-\lambda^2x^2+(a-2\lambda\nu)x+b-\nu^2\bigr).
\end{equation}
Its three roots are $x_1,x_2,x_R$, so Vieta's formula gives
$x_R=\lambda^2-x_1-x_2$.  Reflection of
$R=(x_R,\lambda x_R+\nu)$ gives \eqref{eq:Weierstrass-add}.  Replacing the
chord by the tangent and differentiating
$y^2=x^3+ax+b$ gives \eqref{eq:Weierstrass-double-slope}.  The projective
chord--tangent definition covers vertical chords, points of order two, and
$O$.
\end{proof}

\begin{figure}[htbp]
\centering
\begin{tikzpicture}[x=1.45cm,y=1.25cm]
  \draw[->,black!65] (-1.65,0)--(2.35,0) node[right] {$x$};
  \draw[->,black!65] (0,-1.75)--(0,1.80) node[above] {$y$};
  \draw[modelcurve,domain=-1.30:2.05,samples=100]
      plot ({\x},{.75*sqrt((\x)^3-\x+1)});
  \draw[modelcurve,domain=-1.30:2.05,samples=100]
      plot ({\x},{-.75*sqrt((\x)^3-\x+1)});
  \draw[construction] (-1.38,.954)--(1.65,.418) node[right] {$L(P,Q)$};
  \node[modelpoint,label=above left:$P$] at (-.80,.851) {};
  \node[modelpoint,label=above:$Q$] at (.20,.674) {};
  \node[modelpoint,label=above right:$R$] (WR) at (.656,.593) {};
  \node[sumpoint,label=below right:$P\oplus_W Q$] (WS) at (.656,-.593) {};
  \draw[-{Latex[length=2mm]},thick,blue!65!black] (WR)--(WS)
       node[midway,right] {$y\mapsto-y$};
  \node[anchor=west] at (1.25,1.52) {$O=[0:1:0]$};
\end{tikzpicture}
\caption{The chord--tangent rule on a short Weierstrass cubic.  The chord
through $P,Q$ meets the cubic again at $R$, and reflection gives the sum.}
\label{fig:Weierstrass-addition}
\end{figure}
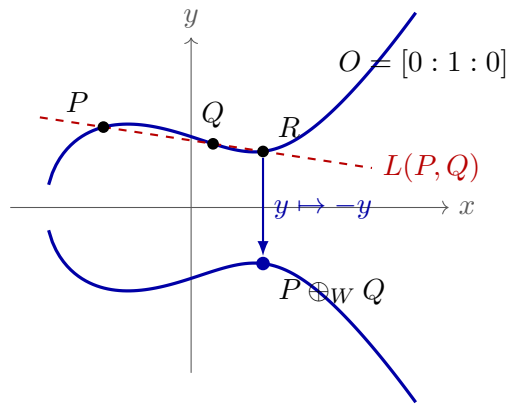

\section{The curves \texorpdfstring{$C_d$}{Cd} in every characteristic}
\label{sec:Cd}

Let
\begin{equation}\label{eq:Cd-affine}
        C_d:\qquad (u^2+u)(v^2+v)=d.
\end{equation}
For a characteristic-uniform study of this family, including its projective
geometry and explicit arithmetic in characteristics $2$ and $3$, see
\cite{Wu2026Cd}.
Its natural completion lies in $\PP^1\times\PP^1$.  Under the Segre embedding
write
\[
 (A:B:C:D)=(U_0V_0:U_1V_0:U_0V_1:U_1V_1),
\]
so that $(A:B:C:D)=(1:u:v:uv)$ on the affine chart.  The completion is
\begin{equation}\label{eq:Cd-projective}
        AD=BC,\qquad D(A+B+C+D)=dA^2.
\end{equation}
It is smooth precisely when
\begin{equation}\label{eq:Cd-smooth}
        d(1-16d)\ne0.
\end{equation}
Indeed, on the affine chart the partial derivatives of
$f=(u^2+u)(v^2+v)-d$ are
\[
 f_u=(2u+1)(v^2+v),\qquad
 f_v=(u^2+u)(2v+1).
\]
If $d\ne0$, the two quadratic factors are nonzero on $C_d$.  In odd
characteristic an affine singularity must therefore have
$u=v=-\tfrac12$, which lies on the curve exactly when $16d=1$.  In
characteristic $2$ the two partial derivatives cannot vanish at a point with
$d\ne0$.  The four boundary points are smooth in reciprocal local charts.
For example, near $(u,v)=(0,\infty)$, with $w=v^{-1}$, the equation is
\[
        (u^2+u)(1+w)-dw^2=0,
\]
whose derivative with respect to $u$ is $1$ at $(0,0)$; the other three
charts follow by these two involutions and by interchanging $u$ and $v$.
Finally, when $d=0$
the affine equation has crossings among the four factors
$u(u+1)v(v+1)$.
In characteristic $2$ this condition is simply $d\ne0$, while in
characteristic $3$ it is $d(1-d)\ne0$.

Set
\begin{equation}\label{eq:Cd-data}
 O=[0:0:1:0],\qquad
 \eta_d(A:B:C:D)=(A:-A-B:C:-C-D).
\end{equation}
Thus $O=(0,\infty)$ and, on the affine chart,
\[
        \eta_d(u,v)=(-u-1,v).
\]
The point $J=\eta_d(O)$ is $[0:0:1:-1]=(-1,\infty)$.

\begin{definition}[Residual-plane law on $C_d$]\label{def:Cd-law}
For $P,Q\in C_d(k)$, let the plane through the length-three scheme $O+P+Q$
meet \eqref{eq:Cd-projective} at the residual point $R$.  Define
\[
        P\oplus_d Q=\eta_d(R).
\]
\end{definition}

\begin{theorem}\label{thm:Cd-group}
If \eqref{eq:Cd-smooth} holds, then $(C_d(k),\oplus_d)$ is an abelian group with
identity $O$.  Its inverse is
\begin{equation}\label{eq:Cd-inverse}
 -(A:B:C:D)=(A:B:-A-C:-B-D),
 \qquad -(u,v)=(u,-v-1).
\end{equation}
This statement is valid in every characteristic.
\end{theorem}

\begin{proof}
The tangent line at $O$ is $B=D=0$.  A plane through this line and an affine
point $(u,v)$ fixes the $v$-coordinate.  Its two intersections away from the
double point $O$ are $(u,v)$ and $(-u-1,v)$, counted with multiplicity.  Thus
\[
          N(P)\cap C_d=2O+P+\eta_d(P).
\]
The four boundary points are
\[
 O=(0,\infty),\quad J=(-1,\infty),\quad
 K=(\infty,0),\quad L=(\infty,-1).
\]
Substitution in \eqref{eq:Cd-projective} gives the same equality at the
boundary: the exceptional plane sections are $3O+J$, $2O+2K$, and
$2O+2L$.  The map $\eta_d$ preserves both defining quadrics and satisfies
$\eta_d^2=1$.  Theorem~\ref{thm:quartic} therefore applies.

Finally, a plane through $O,J,P$ fixes the $u$-coordinate.  Its fourth point is
$(u,-v-1)$, which proves \eqref{eq:Cd-inverse} directly.
\end{proof}

On the affine chart, a plane through $O$ has an equation
\[
        \alpha+\beta u+\gamma uv=0.
\]
Generically it may be written
\begin{equation}\label{eq:Cd-plane}
        v=q+\frac{p}{u}.
\end{equation}
For $P_i=(u_i,v_i)$, $i=1,2$, with $u_1\ne u_2$, the plane through
$O,P_1,P_2$ is obtained from
\begin{equation}\label{eq:pq}
 p=\frac{u_1u_2(v_1-v_2)}{u_2-u_1},
 \qquad
 q=\frac{u_2v_2-u_1v_1}{u_2-u_1}.
\end{equation}
Substitution of \eqref{eq:Cd-plane} into \eqref{eq:Cd-affine} gives the
two-line elimination identity
\begin{align}
u\left((u^2+u)
 \left(\left(q+\frac pu\right)^2+q+\frac pu\right)-d\right)
={}&(q^2+q)u^3+(2pq+p+q^2+q)u^2 \notag\\
 &+(p^2+2pq+p-d)u+p^2.                 \label{eq:Cd-elimination}
\end{align}
The three roots on the affine part are $u_1,u_2,u_R$.  Since the sum is
$\eta_d(R)$, its first coordinate is $u_3=-u_R-1$.  Vieta's formula in
\eqref{eq:Cd-elimination} yields
\begin{equation}\label{eq:Cd-add}
 \boxed{
 u_3=u_1+u_2+\frac{p(2q+1)}{q(q+1)},
 \qquad
 v_3=q-\frac{p}{u_3+1}.}
\end{equation}
Formula \eqref{eq:Cd-add} is used when its displayed denominators are nonzero;
Definition~\ref{def:Cd-law} covers every exceptional and boundary case.

For doubling at an affine point $(u,v)$, choose the tangent member of
\eqref{eq:Cd-plane}.  On its generic domain the parameters satisfy
\begin{equation}\label{eq:Cd-tangent}
 q=v-\frac pu,
 \qquad
 p(u+1)(2v+1)=u(2u+1)(v^2+v),
\end{equation}
and the output is again given by \eqref{eq:Cd-add} with $u_1=u_2=u$.

The derivation does not invert $2$ or $3$.  In characteristic $2$, subtraction
and addition agree, and \eqref{eq:pq}--\eqref{eq:Cd-add} become
\begin{equation}\label{eq:Cd-char2}
 \boxed{
 p=\frac{u_1u_2(v_1+v_2)}{u_1+u_2},\quad
 q=\frac{u_2v_2+u_1v_1}{u_1+u_2},\quad
 u_3=u_1+u_2+\frac{p}{q(q+1)},\quad
 v_3=q+\frac{p}{u_3+1}.}
\end{equation}
In characteristic $3$, where $2=-1$, the last two coordinates read
\begin{equation}\label{eq:Cd-char3}
 \boxed{
 u_3=u_1+u_2+\frac{p(1-q)}{q(q+1)},
 \qquad
 v_3=q-\frac{p}{u_3+1},}
\end{equation}
with $p,q$ as in \eqref{eq:pq}.  Equations \eqref{eq:Cd-char2} and
\eqref{eq:Cd-char3} are the corresponding specializations of
\eqref{eq:Cd-add}.

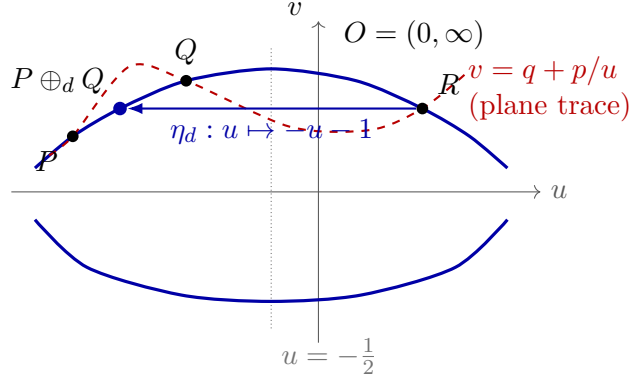
\begin{figure}[htbp]
\centering
\begin{tikzpicture}[x=1.25cm,y=1.05cm]
  \draw[->,black!65] (-3.25,0)--(2.35,0) node[right] {$u$};
  \draw[->,black!65] (0,-1.9)--(0,2.15);
  \node[anchor=south east] at (-.06,2.10) {$v$};
  \node[anchor=west] at (.16,1.98) {$O=(0,\infty)$};
  \draw[guide] (-0.5,-1.72) node[below right] {$u=-\tfrac12$}
      --(-0.5,1.82);
  \draw[modelcurve] plot[smooth] coordinates
    {(-3,.30) (-2.6,.70) (-2.1,1.05) (-1.4,1.40)
     (-.5,1.55) (.4,1.40) (1.1,1.05) (1.6,.70) (2,.30)};
  \draw[modelcurve] plot[smooth] coordinates
    {(-3,-.35) (-2.45,-.95) (-1.45,-1.30) (-.5,-1.38)
     (.45,-1.30) (1.45,-.95) (2,-.35)};
  \draw[construction] plot[smooth] coordinates
    {(-2.85,.46) (-2.6,.70) (-2.0,1.58) (-1.4,1.40)
     (-.3,.82) (.55,.78) (1.1,1.05) (1.55,1.48)};
  \node[modelpoint,label=below left:$P$] at (-2.6,.70) {};
  \node[modelpoint,label=above:$Q$] at (-1.4,1.40) {};
  \node[modelpoint,label=above right:$R$] (R) at (1.1,1.05) {};
  \node[sumpoint,label=above left:$P\oplus_d Q$] (S) at (-2.1,1.05) {};
  \draw[-{Latex[length=2mm]},thick,blue!65!black] (R)--(S)
      node[midway,below] {$\eta_d:u\mapsto-u-1$};
  \node[red!72!black,align=left] at (2.45,1.30)
      {$v=q+p/u$\\(plane trace)};
\end{tikzpicture}
\caption{The residual-plane addition rule on $C_d$ (schematic affine view).  The
upper end of the $v$-axis represents $O=(0,\infty)$.  The plane trace through
$O,P,Q$ has residual point $R$; reflection in $u=-\tfrac12$ gives
$P\oplus_d Q=\eta_d(R)$.}
\label{fig:Cd-addition}
\end{figure}

\section{Edwards curves}

Assume $\operatorname{char}k\ne2,3$, and let
\begin{equation}\label{eq:Edwards-affine}
        E_\delta:\qquad x^2+y^2=1+\delta x^2y^2,
        \qquad \delta(1-\delta)\ne0.
\end{equation}
Its smooth degree-four completion in $\PP^3$ is
\begin{equation}\label{eq:Edwards-projective}
        XY=ZT,\qquad X^2+Y^2=Z^2+\delta T^2,
\end{equation}
with $(X:Y:Z:T)=(x:y:1:xy)$.
The Jacobian criterion gives smoothness under the stated parameter condition.
On the affine chart, a singular point would satisfy
$x(1-\delta y^2)=y(1-\delta x^2)=0$; the curve equation then forces
$\delta=1$.  On $Z=0$, the two gradients are independent when
$\delta\ne0$.  Thus $\delta(1-\delta)\ne0$ is also necessary.

Put
\begin{equation}\label{eq:Edwards-data}
 O=(0:1:1:0),\qquad
 \eta_E(X:Y:Z:T)=(X:-Y:Z:-T).
\end{equation}
Thus $\eta_E(x,y)=(x,-y)$ and $J=\eta_E(O)=(0:-1:1:0)$.

\begin{definition}[Edwards residual-plane law]
Let $H_E(P,Q)$ be the plane through $O,P,Q$, counted with tangent
multiplicities, and write
\[
        H_E(P,Q)\cap E_\delta=O+P+Q+R.
\]
Define $P\oplus_E Q=\eta_E(R)$.
\end{definition}

\begin{theorem}\label{thm:Edwards-group}
The operation $\oplus_E$ makes $E_\delta(k)$ an abelian group with identity
$O$ and inverse
\[
        -(X:Y:Z:T)=(-X:Y:Z:-T),\qquad -(x,y)=(-x,y).
\]
On the affine open set where the denominators do not vanish,
\begin{equation}\label{eq:Edwards-add}
 \boxed{
 (x_1,y_1)\oplus_E(x_2,y_2)=
 \left(
 \frac{x_1y_2+y_1x_2}{1+\delta x_1x_2y_1y_2},
 \frac{y_1y_2-x_1x_2}{1-\delta x_1x_2y_1y_2}
 \right).}
\end{equation}
\end{theorem}

\begin{proof}
The tangent line at $O$ is
\[
        X-T=0,\qquad Y-Z=0.
\]
Put $\ell_1=X-T$ and $\ell_2=Y-Z$.  If
$P=(X_P:Y_P:Z_P:T_P)\ne O$, its normal plane is
\begin{equation}\label{eq:Edwards-normal-plane}
 \ell_2(P)\ell_1-\ell_1(P)\ell_2=0.
\end{equation}
The identity
\[
 \ell_2(P)\ell_1(\eta_E(P))-\ell_1(P)\ell_2(\eta_E(P))
       =2(X_PY_P-Z_PT_P)=0
\]
shows that this plane contains $\eta_E(P)$.  It also contains the tangent
line at $O$, and hence its section contains the length-two subscheme $2O$.
Since a plane section has degree four, its section is
$2O+P+\eta_E(P)$, with the same conclusion when the two last points
coincide.  On the affine chart, this multiplicity statement is visible from
the trace
\[
             (1-y')(x'-x)=0:
\]
the factor $y'=1$ gives $2O$, while $x'=x$ gives the two roots $y'$ and
$-y'$.  To check the boundary, work over $\bar k$ and choose $r^2=\delta$.
The four points with $Z=0$ are
\[
 P_\pm=(0:\pm r:0:1),\qquad Q_\pm=(\pm r:0:0:1).
\]
The planes $Y-Z\pm r(X-T)=0$ cut out $2O+2P_\pm$, respectively, while
$Y-Z=0$ cuts out $2O+Q_++Q_-$.  Finally, the osculating plane is
$N(O):X-T=0$, and
\[
             N(O)\cap E_\delta=3O+J.
\]
Indeed, its affine trace is $x(1-y)=0$: the factor $x$ gives $O+J$, and
$1-y$ gives $2O$.
Thus the normal-plane relation holds for every point.  The map $\eta_E$
preserves both quadrics and is an involution, so
Theorem~\ref{thm:quartic} proves the group axioms.  The plane through $O,J,P$
fixes $y$ and gives the stated inverse.

A plane through $O$ has equation
\[
        aX+b(Y-Z)+cT=0.
\]
For generic affine inputs, coefficients of the plane through $P_1,P_2$ may be
chosen as
\[
\begin{aligned}
 a&=(y_1-1)x_2y_2-x_1y_1(y_2-1),\\
 b&=x_1x_2(y_1-y_2),\\
 c&=x_1(y_2-1)-x_2(y_1-1).
\end{aligned}
\]
The plane equation gives $y=(b-ax)/(b+cx)$.  Substitution into the Edwards
equation, followed by multiplication by $(b+cx)^2$, yields
\begin{align*}
x\bigl(& (c^2-\delta a^2)x^3+2b(\delta a+c)x^2\\
       &+(a^2+(1-\delta)b^2-c^2)x-2b(a+c)\bigr)=0.
\end{align*}
The cubic factor has roots $x_1,x_2,x_R$.  Vieta's formula therefore gives
\[
 x_R=-x_1-x_2-\frac{2b(\delta a+c)}{c^2-\delta a^2},
 \qquad
 -y_R=\frac{ax_R-b}{b+cx_R}.
\]
Using $x_i^2+y_i^2=1+\delta x_i^2y_i^2$ in these two expressions reduces them
to \eqref{eq:Edwards-add}.  Since $\eta_E(R)=(x_R,-y_R)$, this proves the
formula on its stated affine domain.  The projective definition covers the
remaining cases.
\end{proof}

\begin{figure}[htbp]
\centering
\begin{tikzpicture}[x=1.22cm,y=1.08cm]
  \draw[->,black!65] (-2.25,0)--(2.25,0) node[right] {$x$};
  \draw[->,black!65] (0,-2.05)--(0,2.12) node[above] {$y$};
  \draw[modelcurve] plot[smooth cycle,tension=.72] coordinates
    {(0,1.72) (1.15,1.42) (1.62,.78) (1.72,0)
     (1.62,-.78) (1.15,-1.42) (0,-1.72)
     (-1.15,-1.42) (-1.62,-.78) (-1.72,0)
     (-1.62,.78) (-1.15,1.42)};
  \draw[construction] plot[smooth] coordinates
    {(0,1.72) (-.48,1.52) (-1.45,.72) (-1.35,-.92)
     (-.82,-1.52) (.20,-1.18) (1.38,.78)};
  \node[modelpoint,label=above right:$O$] at (0,1.72) {};
  \node[modelpoint,label=above left:$P$] at (-1.45,.72) {};
  \node[modelpoint,label=below left:$Q$] at (-.82,-1.52) {};
  \node[modelpoint,label=above right:$R$] (R) at (1.38,.78) {};
  \node[sumpoint,label=below right:$P\oplus_E Q$] (S) at (1.38,-.78) {};
  \draw[-{Latex[length=2mm]},thick,blue!65!black] (R)--(S)
      node[midway,left] {$\eta_E$};
  \node[red!72!black,align=left] at (-2.55,-1.42)
      {$ax+b(y-1)+cxy=0$\\(plane trace)};
\end{tikzpicture}
\caption{Edwards addition in its affine chart.  The bilinear trace of
$H_E(P,Q)$ gives $R$, and $\eta_E(x,y)=(x,-y)$ gives the sum.}
\label{fig:Edwards-addition}
\end{figure}
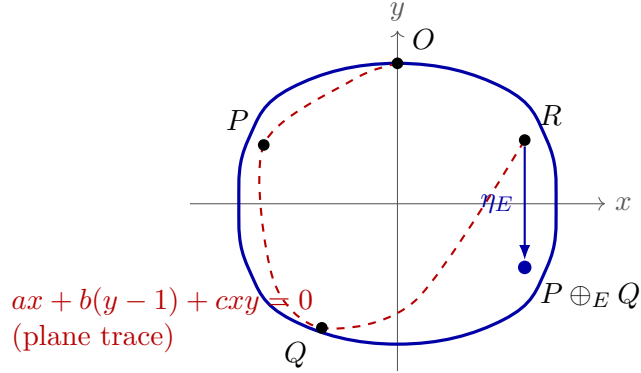

\section{Hessian curves}

Assume $\operatorname{char}k\ne2,3$.  For $\lambda^3\ne1$, let
\begin{equation}\label{eq:Hessian}
 H_\lambda:\qquad X^3+Y^3+Z^3=3\lambda XYZ.
\end{equation}
The gradient equations at a singular point would be
$X^2=\lambda YZ$, $Y^2=\lambda XZ$, and $Z^2=\lambda XY$.  If one coordinate
vanishes, all three vanish; otherwise multiplication gives $\lambda^3=1$.
Conversely, when $\lambda^3=1$, the point
$(1:\lambda^{-1}:1)$ satisfies the equation and all three gradient equations.
Thus the stated condition is exactly the smoothness condition.
The point $O=[1:-1:0]$ is a flex.  The third point on the line $OP$ is
\begin{equation}\label{eq:Hessian-iota}
        \iota_H(X:Y:Z)=(Y:X:Z).
\end{equation}

\begin{definition}[Hessian chord law]
If the line through $P,Q$ meets \eqref{eq:Hessian} once more at $R$, define
\[
        P\oplus_H Q=\iota_H(R).
\]
Tangents are used when $P=Q$.
\end{definition}

\begin{theorem}\label{thm:Hessian-group}
The set $H_\lambda(k)$ is an abelian group under $\oplus_H$, with identity
$O$ and inverse $-P=\iota_H(P)$.  For generic projective inputs,
\begin{align}
 X_3&=Y_1^2X_2Z_2-Y_2^2X_1Z_1,\notag\\
 Y_3&=X_1^2Y_2Z_2-X_2^2Y_1Z_1,\label{eq:Hessian-projective-add}\\
 Z_3&=Z_1^2X_2Y_2-Z_2^2X_1Y_1.\notag
\end{align}
On $Z_1=Z_2=1$ this is
\begin{equation}\label{eq:Hessian-affine-add}
 x_3=\frac{y_1^2x_2-y_2^2x_1}{x_2y_2-x_1y_1},
 \qquad
 y_3=\frac{x_1^2y_2-x_2^2y_1}{x_2y_2-x_1y_1}.
\end{equation}
Generic doubling is
\begin{equation}\label{eq:Hessian-double}
 2(x,y)=\left(
 \frac{y(1-x^3)}{x^3-y^3},
 \frac{x(y^3-1)}{x^3-y^3}
 \right).
\end{equation}
\end{theorem}

\begin{proof}
The equation is symmetric in $X,Y$, and the determinant of the rows
$O,P,\iota_H(P)$ is zero.  Hence these three points are collinear.  The tangent
at $O$ is $X+Y+\lambda Z=0$.  Substitution of
$(X:Y:Z)=(1:-1-\lambda t:t)$ into \eqref{eq:Hessian} gives
$(1-\lambda^3)t^3$, so the tangent meets the cubic at $O$ with multiplicity
three.  Theorem~\ref{thm:cubic} now proves the group axioms.

For the formula, restrict the cubic equation to the line spanned by
$P_1,P_2$.  The resulting binary cubic has the two known zeros corresponding
to $P_1,P_2$; division by these two linear factors and reflection of the third
zero give \eqref{eq:Hessian-projective-add}.  Equivalently, if
$P_3=(X_3:Y_3:Z_3)$ is the tuple in that formula, direct expansion using
$F(P_1)=F(P_2)=0$ gives
\[
 F(P_3)=0,\qquad
 \det\!\begin{pmatrix}
 X_1&Y_1&Z_1\\ X_2&Y_2&Z_2\\ Y_3&X_3&Z_3
 \end{pmatrix}=0,
\]
where $F=X^3+Y^3+Z^3-3\lambda XYZ$.  Thus
$P_1,P_2,\iota_H(P_3)$ are collinear and $P_3=P_1\oplus_H P_2$ on the generic
domain.  For $P=(x,y)$, the tangent line is
\[
 (x^2-\lambda y)X+(y^2-\lambda x)Y+(1-\lambda xy)Z=0.
\]
Its third intersection, followed by $X\leftrightarrow Y$, gives
\eqref{eq:Hessian-double}.  If a displayed tuple vanishes or an affine
denominator is zero, the chord--tangent definition gives the value.
\end{proof}

\begin{figure}[htbp]
\centering
\begin{tikzpicture}[x=1.25cm,y=1.05cm]
  \draw[modelcurve] plot[smooth cycle,tension=.72] coordinates
    {(-1.60,1.25) (-2.18,.70) (-2.02,-.42) (-1.50,-.27)
     (-.78,-.82) (-.20,.094) (-.42,1.02)};
  \draw[modelcurve] plot[smooth,tension=.72] coordinates
    {(-.12,-1.72) (.12,-.68) (.50,.924) (1.18,1.14)
     (2.00,.71) (2.58,.10) (2.90,-1.08)};
  \draw[construction] (-2.55,-.564)--(2.82,.940)
      node[right] {$L(P,Q)$};
  \draw[secondconstruction] (-1.82,1.284)--(2.43,.643)
      node[right] {$L(O,R)$};
  \node[modelpoint,label=below left:$P$] at (-1.50,-.27) {};
  \node[modelpoint,label=above:$Q$] at (-.20,.094) {};
  \node[modelpoint,label=above right:$R$] at (2.00,.71) {};
  \node[modelpoint,label=above left:$O$] at (-1.60,1.25) {};
  \node[sumpoint] (HS) at (.50,.924) {};
  \node[above=2pt of HS] {$\iota_H(R)=P\oplus_H Q$};
  \node[blue!65!black] at (0,-1.53) {$H_\lambda$};
\end{tikzpicture}
\caption{The two-line rule on the Hessian cubic (projective schematic).
The chord $L(P,Q)$ gives $R$; the line $L(O,R)$ gives the third point
$\iota_H(R)=P\oplus_H Q$.}
\label{fig:Hessian-addition}
\end{figure}

\section{Jacobi quartics}\label{sec:Jacobi-quartic}

Assume $\operatorname{char}k\ne2,3$, and let
\begin{equation}\label{eq:Jacobi-quartic}
        J_a:\qquad y^2=x^4+2ax^2+1,
        \qquad a^2\ne1.
\end{equation}
Use the degree-four embedding
\[
        (x,y)\longmapsto(Z:X:W:Y)=(1:x:x^2:y),
\]
whose projective closure is the smooth intersection
\begin{equation}\label{eq:Jacobi-quartic-projective}
        X^2=ZW,\qquad Y^2=W^2+2aZW+Z^2.
\end{equation}
An affine singularity would have $y=0$ and
$4x(x^2+a)=0$.  The curve equation excludes $x=0$, while $x^2=-a$ gives
$1-a^2=0$.  At $Z=0$ the two points satisfy $X=0$ and $Y=\pm W$, and the
two quadric gradients are independent.  Thus $a^2\ne1$ is precisely the
smoothness condition in the present characteristic.
Set
\begin{equation}\label{eq:Jacobi-quartic-data}
 O=(1:0:0:1),\qquad
 \eta_J(Z:X:W:Y)=(Z:-X:W:Y).
\end{equation}
Here $\eta_J(O)=O$.

\begin{definition}[Jacobi-quartic residual-plane law]
If the plane through $O,P,Q$ has residual intersection $R$ with
\eqref{eq:Jacobi-quartic-projective}, define
\[
        P\oplus_J Q=\eta_J(R).
\]
\end{definition}

\begin{theorem}\label{thm:Jacobi-quartic-group}
The operation $\oplus_J$ makes $J_a(k)$ an abelian group with identity $O$ and
inverse $-(x,y)=(-x,y)$.  On its generic affine domain, put
$r=1-(x_1x_2)^2$.  Then
\begin{align}
 x_3&=\frac{x_1y_2+y_1x_2}{r},\label{eq:Jacobi-quartic-x}\\
 y_3&=\frac{(1+(x_1x_2)^2)(y_1y_2+2ax_1x_2)
              +2x_1x_2(x_1^2+x_2^2)}{r^2}.
                                                        \label{eq:Jacobi-quartic-y}
\end{align}
\end{theorem}

\begin{proof}
The tangent line at $O$ is
\[
        W=0,\qquad Y-Z-aW=0.
\]
Put $\ell_1=W$ and $\ell_2=Y-Z-aW$.  For $P\ne O$, the plane
\begin{equation}\label{eq:Jacobi-normal-plane}
             \ell_2(P)\ell_1-\ell_1(P)\ell_2=0
\end{equation}
contains the tangent line at $O$, as well as $P$ and $\eta_J(P)$, because
$\ell_1$ and $\ell_2$ are unchanged by $X\mapsto-X$.  Its section is therefore
$2O+P+\eta_J(P)$.  Apart from the member $W=0$, its affine trace has the form
$y-1+n x^2=0$; substitution in the quartic shows directly that the remaining
roots are $x$ and $-x$, with their scheme multiplicities.  The plane $W=0$
cuts out $2O+2(1:0:0:-1)$.  The two boundary points are
$B_\pm=(0:0:1:\pm1)$, and the planes $Y-Z=\pm W$ cut out
$2O+2B_\pm$, respectively.  The osculating plane
\[
             N(O):Y-Z-aW=0
\]
cuts out $4O$, since its affine substitution gives
$(a^2-1)x^4$.  This proves the normal-plane relation for every point.  The
map $\eta_J$ preserves both quadrics and is an involution.
Theorem~\ref{thm:quartic} gives the group law and, since $\eta_J(O)=O$, the
stated inverse.

A general plane through $O$ has affine trace
\begin{equation}\label{eq:Jacobi-parabola}
        y=1+mx+nx^2.
\end{equation}
Substitution gives
\begin{align}
 &(1+mx+nx^2)^2-(x^4+2ax^2+1)\notag\\
 &\qquad=x\bigl(2m+(m^2+2n-2a)x+2mnx^2+(n^2-1)x^3\bigr).
                                                        \label{eq:Jacobi-elimination}
\end{align}
Choose $m,n$ so that \eqref{eq:Jacobi-parabola} passes through $P_1,P_2$.
On the generic domain $x_1x_2(x_2-x_1)\ne0$, these coefficients are
\begin{equation}\label{eq:Jacobi-mn}
 m=\frac{x_2^2(y_1-1)-x_1^2(y_2-1)}
          {x_1x_2(x_2-x_1)},\qquad
 n=\frac{x_1(y_2-1)-x_2(y_1-1)}
          {x_1x_2(x_2-x_1)}.
\end{equation}
If $x_R$ is the third nonzero root in \eqref{eq:Jacobi-elimination}, then
$x_3=-x_R$, and Vieta gives
\begin{equation}\label{eq:Jacobi-mn-add}
        x_3=x_1+x_2+\frac{2mn}{n^2-1},\qquad
        y_3=1-mx_3+nx_3^2.
\end{equation}
Substitution of \eqref{eq:Jacobi-mn} into
\eqref{eq:Jacobi-mn-add}, followed by the replacements
$y_i^2=x_i^4+2ax_i^2+1$, gives
\eqref{eq:Jacobi-quartic-x}--\eqref{eq:Jacobi-quartic-y}.  The projective
plane rule handles all exceptional cases.
\end{proof}

\begin{remark}
The parabola in \eqref{eq:Jacobi-parabola} has the same formal shape as the
parabola used by Zwegers for a Weierstrass cubic \cite{Zwegers2024}.  Here it
is the affine trace of a hyperplane section of the embedded Jacobi quartic;
the factor $x$ in \eqref{eq:Jacobi-elimination} records the fixed intersection
at $O$, and the remaining cubic factor contains the two inputs and the
residual point.
\end{remark}

\begin{figure}[htbp]
\centering
\begin{tikzpicture}[x=1.45cm,y=1.05cm]
  \draw[->,black!65] (-2.55,0)--(2.55,0) node[right] {$x$};
  \draw[->,black!65] (0,-2.05)--(0,2.12) node[above] {$y$};
  \draw[modelcurve,domain=-1.32:1.32,samples=90]
      plot ({1.8*\x},{.9*sqrt((\x)^4+1)});
  \draw[modelcurve,domain=-1.32:1.32,samples=90]
      plot ({1.8*\x},{-.9*sqrt((\x)^4+1)});
  \draw[construction,domain=-1.27:1.27,samples=70]
      plot ({1.8*\x},{.9*(1-.05*\x+.25*(\x)^2)});
  \node[modelpoint,label=below right:$O$] at (0,.9) {};
  \node[modelpoint,label=above left:$P$] at (-1.492,1.092) {};
  \node[modelpoint,label=below right:$Q$] at (.400,.901) {};
  \node[modelpoint,label=above right:$R$] (R) at (1.044,.950) {};
  \node[sumpoint,label=above left:$P\oplus_J Q$] (S) at (-1.044,.950) {};
  \draw[-{Latex[length=2mm]},thick,blue!65!black] (R)--(S)
       node[midway,below] {$\eta_J:x\mapsto-x$};
  \node[red!72!black,align=left] at (2.82,1.55)
       {$y=1+mx+nx^2$\\(plane trace)};
\end{tikzpicture}
\caption{Jacobi-quartic addition.  The plotted example has $a=0$; the
construction is general.  A parabola through $O,P,Q$ gives the residual point
$R$, whose reflection in the $y$-axis is the sum.}
\label{fig:Jacobi-quartic-addition}
\end{figure}
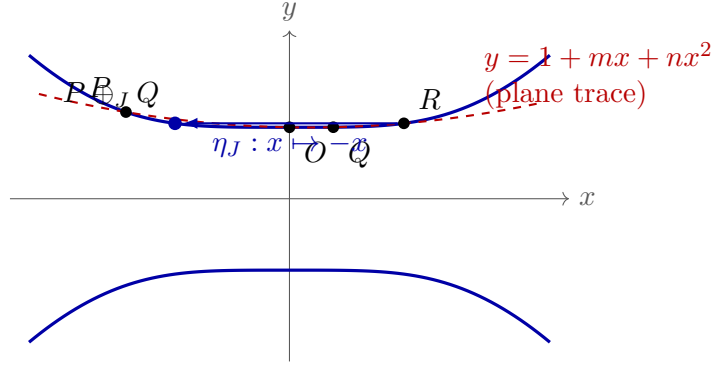

\section{Jacobi intersections}

Assume $\operatorname{char}k\ne2,3$, and let
\begin{equation}\label{eq:Jacobi-intersection}
 I_\alpha:\qquad s^2+c^2=1,\qquad \alpha s^2+n^2=1,
 \qquad \alpha(1-\alpha)\ne0.
\end{equation}
The twisted generalization of this intersection model and its explicit
arithmetic were developed by Feng, Nie, and Wu \cite{FengNieWu2013}.
Its projective closure in coordinates $(S:C:N:Z)$ is the smooth intersection
\begin{equation}\label{eq:Jacobi-intersection-projective}
        S^2+C^2=Z^2,\qquad \alpha S^2+N^2=Z^2.
\end{equation}
The Jacobian criterion gives the stated smoothness condition.  If
$\alpha=0$, the points with $N=Z=0$ and $S^2+C^2=0$ are singular.  If
$\alpha\ne0$ and the two quadric gradients are dependent, their $C$- and
$N$-components force $C=N=0$; the two equations then force $\alpha=1$.
Conversely, for $\alpha=1$ the points with $C=N=0$ and $S^2=Z^2$ are
singular.
Put
\begin{equation}\label{eq:Jacobi-intersection-data}
        O=(0:1:1:1),\qquad
        \eta_I(S:C:N:Z)=(-S:C:N:Z).
\end{equation}

\begin{definition}[Jacobi-intersection residual-plane law]
If the plane through $O,P,Q$ has residual intersection $R$ with
\eqref{eq:Jacobi-intersection-projective}, define
\[
        P\oplus_I Q=\eta_I(R).
\]
\end{definition}

\begin{theorem}\label{thm:Jacobi-intersection-group}
The operation $\oplus_I$ makes $I_\alpha(k)$ an abelian group with identity
$O$ and inverse $-(s,c,n)=(-s,c,n)$.  On the affine open set where
\[
        \Delta=1-\alpha s_1^2s_2^2
\]
is nonzero, the sum is
\begin{align}
 s_3&=\frac{s_1c_2n_2+s_2c_1n_1}{\Delta},\notag\\
 c_3&=\frac{c_1c_2-s_1s_2n_1n_2}{\Delta},
                                                        \label{eq:Jacobi-intersection-add}\\
 n_3&=\frac{n_1n_2-\alpha s_1s_2c_1c_2}{\Delta}.
\notag
\end{align}
\end{theorem}

\begin{proof}
The tangent line at $O$ is $C=Z$, $N=Z$.  A plane through it is a linear
combination of $\ell_1=C-Z$ and $\ell_2=N-Z$.  For $P\ne O$, its normal
plane is
\begin{equation}\label{eq:Jacobi-intersection-normal-plane}
             \ell_2(P)\ell_1-\ell_1(P)\ell_2=0.
\end{equation}
Because $\eta_I$ changes only $S$, this plane contains $P$ and $\eta_I(P)$;
it also contains the tangent line at $O$.  Hence its section is
$2O+P+\eta_I(P)$ whenever these last two points are distinct.  The remaining
fixed points are verified by the three plane sections
\[
\begin{array}{c@{\qquad}c}
 N-Z=0:&2O+2(0:-1:1:1),\\
 C-Z=0:&2O+2(0:1:-1:1),\\
 N-C=0:&2O+2(0:-1:-1:1).
\end{array}
\]
At a boundary point $Z=0$, one has $S\ne0$, so $P\ne\eta_I(P)$ and the same
degree-four argument applies.  The osculating plane is
\[
            N(O):(N-Z)-\alpha(C-Z)=0.
\]
On the affine chart its substitution reduces to
$\alpha(\alpha-1)(c-1)^2=0$.  Since
$(c-1)(c+1)=-s^2$ and $c+1$ is a unit at $O$, the function $c-1$ has order
two there; hence the section is $4O$.  Thus the normal-plane relation holds
everywhere.  The map $\eta_I$
preserves the two quadrics and is an involution.  Theorem~\ref{thm:quartic}
proves the group axioms; since $\eta_I(O)=O$, the same involution is group
inversion.

A general plane through $O$ has equation
\[
        \ell S+m(C-Z)+r(N-Z)=0.
\]
To verify the affine formula, set
\[
\begin{aligned}
 A&=s_1c_2n_2+s_2c_1n_1,\\
 B&=c_1c_2-s_1s_2n_1n_2,\\
 C_0&=n_1n_2-\alpha s_1s_2c_1c_2.
\end{aligned}
\]
Using $c_i^2=1-s_i^2$ and $n_i^2=1-\alpha s_i^2$, expansion gives
\begin{equation}\label{eq:Jacobi-intersection-identities}
        A^2+B^2=\Delta^2,\qquad
        \alpha A^2+C_0^2=\Delta^2,
\end{equation}
and
\begin{equation}\label{eq:Jacobi-intersection-determinant}
 \det\!\begin{pmatrix}
 s_1&c_1-1&n_1-1\\
 s_2&c_2-1&n_2-1\\
 -A&B-\Delta&C_0-\Delta
 \end{pmatrix}=0.
\end{equation}
The two identities in \eqref{eq:Jacobi-intersection-identities} show that
$P_3=(A/\Delta,B/\Delta,C_0/\Delta)$ lies on both quadrics.
Equation \eqref{eq:Jacobi-intersection-determinant} says that
$O,P_1,P_2,\eta_I(P_3)$ are coplanar.  Hence $P_3$ is the residual-plane sum,
which proves \eqref{eq:Jacobi-intersection-add}.  The projective rule applies
when $\Delta=0$.
\end{proof}

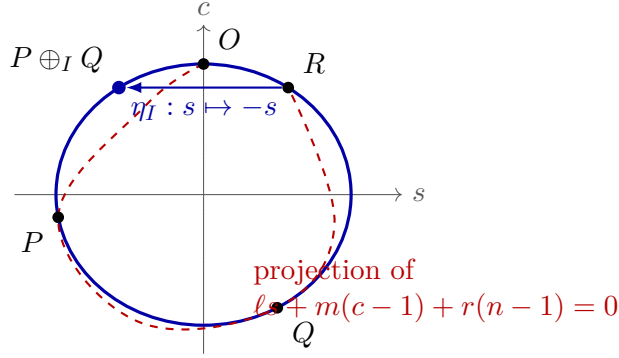
\begin{figure}[htbp]
\centering
\begin{tikzpicture}[x=1.22cm,y=1.08cm]
  \draw[->,black!65] (-2.05,0)--(2.15,0) node[right] {$s$};
  \draw[->,black!65] (0,-1.95)--(0,2.08) node[above] {$c$};
  \draw[guide] (0,-1.72)--(0,1.78);
  \draw[modelcurve] (0,0) circle[radius=1.60];
  \draw[construction] plot[smooth,tension=.65] coordinates
    {(0,1.60) (-.62,1.10) (-1.576,-.278) (-.70,-1.56)
     (.80,-1.386) (1.42,-.35) (.918,1.311)};
  \node[modelpoint,label=above right:$O$] at (0,1.60) {};
  \node[modelpoint,label=below left:$P$] at (-1.576,-.278) {};
  \node[modelpoint,label=below right:$Q$] at (.80,-1.386) {};
  \node[modelpoint,label=above right:$R$] (R) at (.918,1.311) {};
  \node[sumpoint,label=above left:$P\oplus_I Q$] (S) at (-.918,1.311) {};
  \draw[-{Latex[length=2mm]},thick,blue!65!black] (R)--(S)
       node[midway,below] {$\eta_I:s\mapsto-s$};
  \node[red!72!black,align=left] at (2.52,-1.18)
       {projection of\\$\ell s+m(c-1)+r(n-1)=0$};
\end{tikzpicture}
\caption{Jacobi-intersection addition, projected to the $(s,c)$-circle.
The $n$-coordinates are suppressed.  The dashed trace represents the plane
through $O,P,Q$; after finding $R$, the involution preserves $c,n$ and changes
the sign of $s$.}
\label{fig:Jacobi-intersection-addition}
\end{figure}

\section{Huff curves}

Assume $\operatorname{char}k\ne2,3$.  Let $a,b\in k$ satisfy
$ab(a^2-b^2)\ne0$, and consider the smooth plane cubic
\begin{equation}\label{eq:Huff}
 \mathcal H_{a,b}:\qquad
        aX(Y^2-Z^2)=bY(X^2-Z^2).
\end{equation}
Generalized Huff curves containing this model, together with explicit group
arithmetic, were studied by Wu and Feng \cite{WuFeng2012}.
The parameter condition is exactly the smoothness condition.  To see this,
write the left side minus the right side as $F$.  Its partial derivatives are
\[
\begin{aligned}
 F_X&=a(Y^2-Z^2)-2bXY,\\
 F_Y&=2aXY-b(X^2-Z^2),\\
 F_Z&=2Z(bY-aX).
\end{aligned}
\]
If $a=0$ or $b=0$, the cubic is reducible and has singular intersection
points between its components.  Assume therefore that $ab\ne0$.
When $Z=0$, a point with $X=0$ or $Y=0$ has a nonzero partial derivative.  If
$XY\ne0$, the equation $F=0$ gives $aY=bX$, and then
$F_X=-aY^2\ne0$.  Thus there is no singular point at infinity when $ab\ne0$.
When $Z\ne0$, put $Z=1$.  The equation $F_Z=0$ gives $by=ax$, and substitution
in $F_X=F_Y=0$ gives
\[
 (a^2-2b^2)x^2=b^2,\qquad
 (2a^2-b^2)x^2=-b^2.
\]
Since $\operatorname{char}k\ne3$, these equations have a common solution over
an algebraic closure precisely when $a^2=b^2$.  Conversely, if $a^2=b^2$,
choose $x^2=-1$ and put $y=(a/b)x$; all three partial derivatives vanish at
$(x:y:1)$.  Euler's identity $3F=XF_X+YF_Y+ZF_Z$ then shows that this point
lies on the cubic.
The point $O=[0:0:1]$ is a flex.  The third point on the line $OP$ is
\begin{equation}\label{eq:Huff-iota}
        \iota_{\mathcal H}(X:Y:Z)=(X:Y:-Z),
\end{equation}
which is $(-x,-y)$ on the affine chart.

\begin{definition}[Huff chord law]
If the chord or tangent through $P,Q$ has third intersection $R$ with
\eqref{eq:Huff}, define
\[
        P\oplus_{\mathcal H}Q=\iota_{\mathcal H}(R).
\]
\end{definition}

\begin{theorem}\label{thm:Huff-group}
The operation $\oplus_{\mathcal H}$ makes $\mathcal H_{a,b}(k)$ an abelian
group with identity $O$ and inverse $-(x,y)=(-x,-y)$.  On its generic affine
domain,
\begin{equation}\label{eq:Huff-affine-add}
 \boxed{
 x_3=\frac{(x_1+x_2)(1+y_1y_2)}
              {(1+x_1x_2)(1-y_1y_2)},
 \qquad
 y_3=\frac{(y_1+y_2)(1+x_1x_2)}
              {(1-x_1x_2)(1+y_1y_2)}.}
\end{equation}
One homogeneous form of the formula is
\begin{align}
X_3={}&(X_1Z_2+X_2Z_1)(Y_1Y_2+Z_1Z_2)^2
       (Z_1Z_2-X_1X_2),\notag\\
Y_3={}&(Y_1Z_2+Y_2Z_1)(X_1X_2+Z_1Z_2)^2
       (Z_1Z_2-Y_1Y_2),\label{eq:Huff-projective-add}\\
Z_3={}&(Z_1^2Z_2^2-X_1^2X_2^2)
       (Z_1^2Z_2^2-Y_1^2Y_2^2).\notag
\end{align}
\end{theorem}

\begin{proof}
At $O$ the tangent is $aX-bY=0$.  Substituting $(X:Y:Z)=(bt:at:1)$ into
\eqref{eq:Huff} gives $ab(a^2-b^2)t^3$, so $O$ is a flex.  The curve is
invariant under $Z\mapsto-Z$, and $O,P,(X:Y:-Z)$ are collinear.  Therefore
Theorem~\ref{thm:cubic} proves all group axioms on the Huff cubic itself.

For two affine points, put
\[
\begin{aligned}
 A_x&=(x_1+x_2)(1+y_1y_2),&
 D_x&=(1+x_1x_2)(1-y_1y_2),\\
 A_y&=(y_1+y_2)(1+x_1x_2),&
 D_y&=(1-x_1x_2)(1+y_1y_2).
\end{aligned}
\]
On the open set $D_xD_y\ne0$, define $x_3=A_x/D_x$ and $y_3=A_y/D_y$.
Expansion, followed by the two substitutions
$ax_i(y_i^2-1)=by_i(x_i^2-1)$, gives
\begin{align}
 aA_xD_x(A_y^2-D_y^2)&=bA_yD_y(A_x^2-D_x^2),
                                                        \label{eq:Huff-check1}\\
 \det\!\begin{pmatrix}
 x_1&y_1&1\\ x_2&y_2&1\\ -A_x/D_x&-A_y/D_y&1
 \end{pmatrix}&=0.                                  \label{eq:Huff-check2}
\end{align}
The first identity places $(x_3,y_3)$ on the Huff curve, and the second places
$P_1,P_2,(-x_3,-y_3)$ on one line.  Since $(-x_3,-y_3)$ is the reflected
residual point, this proves \eqref{eq:Huff-affine-add}.  Taking the common
homogeneous denominator $D_xD_y$ gives \eqref{eq:Huff-projective-add}.
The chord--tangent definition covers base pairs of the displayed homogeneous
tuple.
\end{proof}

\begin{figure}[htbp]
\centering
\begin{tikzpicture}[x=1.25cm,y=1.0cm]
  \draw[->,black!65] (-2.45,0)--(2.52,0) node[right] {$x$};
  \draw[->,black!65] (0,-2.55)--(0,2.62) node[above] {$y$};
  \draw[modelcurve,domain=-2.25:-.08,samples=110]
    plot ({\x},{.65*((\x)^2-1+sqrt((\x)^4-(\x)^2+1))/\x});
  \draw[modelcurve,domain=.08:2.25,samples=110]
    plot ({\x},{.65*((\x)^2-1+sqrt((\x)^4-(\x)^2+1))/\x});
  \draw[construction,domain=-1.65:1.55]
    plot ({\x},{.65*(1.06157*\x-.15790)})
    node[right] {$L(P,Q)$};
  \draw[secondconstruction,domain=-1.55:1.55]
    plot ({\x},{.57295*\x}) node[right] {$L(O,R)$};
  \node[modelpoint,label=above left:$O$] at (0,0) {};
  \node[modelpoint,label=below left:$P$] at (-1.20,-.931) {};
  \node[modelpoint,label=above left:$Q$] at (.30,.104) {};
  \node[modelpoint,label=above right:$R$] at (.880,.504) {};
  \node[sumpoint] (HuffS) at (-.880,-.504) {};
  \node[below left=2pt and -2pt of HuffS]
      {$\iota_{\mathcal H}(R)=P\oplus_{\mathcal H}Q$};
  \node[blue!65!black] at (1.65,-1.55) {$\mathcal H_{1,2}$};
\end{tikzpicture}
\caption{Huff addition on the representative real curve
$\mathcal H_{1,2}$.  The chord gives $R$, and the line through $O$ and $R$
gives the centrally opposite point $(-x_R,-y_R)$, which is the sum.}
\label{fig:Huff-addition}
\end{figure}
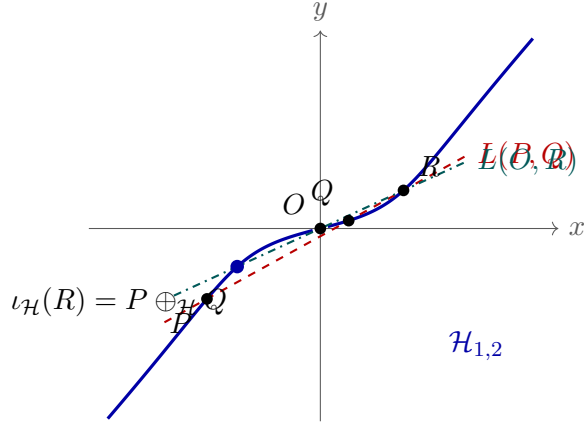

\section{A symmetric QRT biquadratic}\label{sec:QRT}

Assume $\operatorname{char}k\ne2,3$, and consider
\begin{equation}\label{eq:QRT-affine}
  \QQRT_{\alpha,\beta,\gamma}:
  \qquad
  x^2y^2+\alpha(x^2+y^2)+\beta xy+\gamma=0.
\end{equation}
The $C_d$ equation of Section~\ref{sec:Cd} is another symmetric biquadratic.
Its characteristic-uniform treatment and its extension to a broader symmetric
$(2,2)$ family are given in \cite{Wu2026Cd}.
In $\PP^1\times\PP^1$, with coordinates
$((X_0:X_1),(Y_0:Y_1))$, its bihomogeneous equation is
\begin{align}
 X_1^2Y_1^2
 &+\alpha(X_1^2Y_0^2+X_0^2Y_1^2)
 +\beta X_0X_1Y_0Y_1+\gamma X_0^2Y_0^2=0.
                                                        \label{eq:QRT-bihomogeneous}
\end{align}
Under the Segre embedding
\[
 (A:B:C:D)=(X_0Y_0:X_1Y_0:X_0Y_1:X_1Y_1),
\]
this completion is the intersection of two quadrics
\begin{equation}\label{eq:QRT-projective}
 AD=BC,\qquad
 D^2+\alpha(B^2+C^2)+\beta AD+\gamma A^2=0.
\end{equation}

\begin{proposition}\label{prop:QRT-smoothness}
The curve \eqref{eq:QRT-projective} is smooth precisely when
\begin{equation}\label{eq:QRT-discriminant}
 \Delta_{\!Q}=
 \alpha\gamma
 \bigl((2\alpha+\beta)^2-4\gamma\bigr)
 \bigl((2\alpha-\beta)^2-4\gamma\bigr)\ne0.
\end{equation}
When this condition holds, it is a genus-one normal quartic.
\end{proposition}

\begin{proof}
Write the affine left side of \eqref{eq:QRT-affine} as $F$.  Its partial
derivatives are
\[
 F_x=2x(y^2+\alpha)+\beta y,\qquad
 F_y=2y(x^2+\alpha)+\beta x.
\]
If $\alpha\ne0$ and $F_x=F_y=0$, then
$xF_x-yF_y=2\alpha(x^2-y^2)=0$, so $y=x$ or $y=-x$ over an algebraic
closure.  The point $x=y=0$ is singular precisely when $\gamma=0$.  For a
nonzero solution, the case $y=x$ is singular precisely when
$4\gamma=(2\alpha+\beta)^2$, and the case $y=-x$ precisely when
$4\gamma=(2\alpha-\beta)^2$.

It remains to inspect the boundary.  Near $x=\infty$, put $s=1/x$.  Equation
\eqref{eq:QRT-bihomogeneous} becomes
\[
 y^2+\alpha+\beta sy+s^2(\alpha y^2+\gamma)=0.
\]
At $s=0$ its points satisfy $y^2=-\alpha$, and the derivative with respect to
$y$ is $2y$.  Thus the boundary is smooth there exactly when $\alpha\ne0$;
the calculation at $y=\infty$ is identical, and the bihomogeneous equation has
no point with $x=y=\infty$.  This proves
\eqref{eq:QRT-discriminant}.  A smooth complete intersection of two quadrics
in $\PP^3$ has degree four and genus one.
\end{proof}

A group law requires a rational base point.  Suppose henceforth that
\eqref{eq:QRT-discriminant} holds and fix
$O\in\QQRT_{\alpha,\beta,\gamma}(k)$.  For every $P\in\QQRT$, define
\[
       N_Q(P)=\Span(2O+P).
\]
For $P\ne O$, this is the plane spanned by the tangent line $T_O\QQRT$ and
$P$; for $P=O$, it is the osculating plane $\Span(3O)$.  Define $\eta_Q(P)$
as the residual point in
\begin{equation}\label{eq:QRT-eta}
 N_Q(P)\cap\QQRT=2O+P+\eta_Q(P).
\end{equation}
The plane and its residual point are defined over $k$ whenever $P$ is
$k$-rational.

\begin{definition}[QRT residual-plane law]\label{def:QRT-law}
For $P,Q\in\QQRT_{\alpha,\beta,\gamma}(k)$, let
\begin{equation}\label{eq:QRT-first-plane}
 H_Q(P,Q)\cap\QQRT=O+P+Q+R,\qquad
 H_Q(P,Q)=\Span(O+P+Q).
\end{equation}
Define
\[
             P\oplus_{\mathcal Q} Q=\eta_Q(R).
\]
Repeated points in \eqref{eq:QRT-eta} and \eqref{eq:QRT-first-plane} denote
the corresponding length subschemes, including tangent and osculating
multiplicities.
\end{definition}

\begin{theorem}\label{thm:QRT-group}
Definition~\ref{def:QRT-law} makes
$\QQRT_{\alpha,\beta,\gamma}(k)$ an abelian group under
$\oplus_{\mathcal Q}$, with identity $O$.
If $J=\eta_Q(O)$, the inverse of $P$ is the residual point $P^{-}$ in
\[
       \Span(O+J+P)\cap\QQRT=O+J+P+P^{-}.
\]
\end{theorem}

\begin{proof}
For $P\ne O$, the coefficients of the plane through $T_O\QQRT$ and $P$ are
rational functions of the coordinates of $P$.  Its residual intersection
therefore defines a rational self-map of $\QQRT$ over $k$.  A rational map
from a smooth projective curve to a projective curve extends across every
point; as $P$ tends to $O$, the length-three scheme $2O+P$ tends to $3O$, so
the extension has the value prescribed by $N_Q(O)=\Span(3O)$.  Thus $\eta_Q$
is a morphism defined over $k$.  If
$N_Q(P)\cap\QQRT=2O+P+P'$, then the same plane contains
$2O+P'$.  By the uniqueness of the plane spanned by a length-three
subscheme, it is $N_Q(P')$, whose other residual point is $P$.  Consequently
$\eta_Q(P')=P$, including $P=O$ and $P'=O$.  Thus $\eta_Q$ is an involution and
\eqref{eq:QRT-eta} is exactly the normal-plane hypothesis of
Theorem~\ref{thm:quartic}.  That theorem gives commutativity, the identity,
the displayed inverse, and associativity by comparing the two products of
three QRT plane sections.
\end{proof}

The construction is explicit when a boundary base point is rational.  If
$r\in k^\times$ satisfies $r^2=-\alpha$, take
\begin{equation}\label{eq:QRT-origin}
             O=(\infty,r)=[0:1:0:r].
\end{equation}
For the two quadrics in \eqref{eq:QRT-projective}, the tangent line at $O$ is
\begin{equation}\label{eq:QRT-tangent}
       C-rA=0,\qquad \beta A-2rB+2D=0.
\end{equation}
Put $L_1=C-rA$ and $L_2=\beta A-2rB+2D$.  If the first plane has produced
$R\ne O$, the second plane in the addition rule is
\begin{equation}\label{eq:QRT-second-plane}
       L_2(R)L_1-L_1(R)L_2=0;
\end{equation}
its section is $2O+R+(P\oplus_{\mathcal Q}Q)$.  If $R=O$, the second plane is
$N_Q(O)=\Span(3O)$.  On the affine chart, every projective plane has a
bilinear trace
\begin{equation}\label{eq:QRT-bilinear-trace}
             a+bx+cy+dxy=0.
\end{equation}
Thus QRT addition is computed by two successive eliminations of a biquadratic
with a bilinear equation: the first gives $R$, and the plane through the
tangent line in \eqref{eq:QRT-second-plane} gives the sum.  For the choice
\eqref{eq:QRT-origin}, a trace through $O$ satisfies $b+dr=0$.

The two fiber involutions of the symmetric biquadratic are
\begin{align}
 \sigma_x(x,y)&=\left(-x-\frac{\beta y}{y^2+\alpha},\,y\right),
                                                        \label{eq:QRT-sigmax}\\
 \sigma_y(x,y)&=\left(x,\,-y-\frac{\beta x}{x^2+\alpha}\right).
                                                        \label{eq:QRT-sigmay}
\end{align}
They are obtained by taking the other root of \eqref{eq:QRT-affine} with one
coordinate fixed, and extend projectively across the displayed denominators.
Their composition is the symmetric QRT evolution of
\cite{QRT1988,QRT1989}.  The group operation
$\oplus_{\mathcal Q}$ is instead defined by the two plane sections
\eqref{eq:QRT-first-plane} and \eqref{eq:QRT-second-plane}.

\begin{figure}[htbp]
\centering
\begin{tikzpicture}[x=1.0cm,y=1.0cm]
  \begin{scope}[shift={(-3.35,0)}]
    \draw[modelcurve] plot[smooth cycle,tension=.75] coordinates
      {(0,1.45) (1.15,1.10) (1.55,.25) (1.25,-.72)
       (.15,-1.38) (-1.05,-1.05) (-1.55,-.20) (-1.25,.82)};
    \draw[construction] plot[smooth,tension=.68] coordinates
      {(0,1.45) (-.72,1.02) (-1.42,.18) (-.72,-1.22)
       (.25,-1.32) (1.43,.30) (1.00,1.08)};
    \node[modelpoint,label=above:$O$] at (0,1.45) {};
    \node[modelpoint,label=left:$P$] at (-1.42,.18) {};
    \node[modelpoint,label=below:$Q$] at (-.72,-1.22) {};
    \node[modelpoint,label=right:$R$] at (1.43,.30) {};
    \node[below=1.73cm,align=center] at (0,0)
      {first plane:\\$O+P+Q+R$};
  \end{scope}
  \draw[-{Latex[length=2.4mm]},very thick] (-1.25,0)--(.55,0)
       node[midway,above] {$R$};
  \begin{scope}[shift={(3.25,0)}]
    \draw[modelcurve] plot[smooth cycle,tension=.75] coordinates
      {(0,1.45) (1.15,1.10) (1.55,.25) (1.25,-.72)
       (.15,-1.38) (-1.05,-1.05) (-1.55,-.20) (-1.25,.82)};
    \draw[guide] (-1.05,1.45)--(1.05,1.45) node[right] {$T_O\QQRT$};
    \draw[secondconstruction]
      (0,1.45) .. controls (1.05,1.45) and (1.72,.83) .. (1.43,.30)
      .. controls (1.05,-.72) and (-.18,-1.47) .. (-1.05,-1.05)
      .. controls (-1.72,-.18) and (-1.02,1.45) .. (0,1.45);
    \node[modelpoint,label=above:$O$] at (0,1.45) {};
    \node[modelpoint,label=right:$R$] at (1.43,.30) {};
    \node[sumpoint,label=left:$P\oplus_{\mathcal Q}Q$] at (-1.05,-1.05) {};
    \node[below=1.73cm,align=center] at (0,0)
      {second plane:\\$2O+R+(P\oplus_{\mathcal Q}Q)$};
  \end{scope}
\end{tikzpicture}
\caption{The two-stage QRT residual-plane rule (schematic).  Red denotes the
bilinear trace of $H_Q(P,Q)$; teal denotes the trace of the plane through
$T_O\QQRT$ and $R$.  Tangency at $O$ accounts for its multiplicity two in the
second panel.}
\label{fig:QRT-addition}
\end{figure}
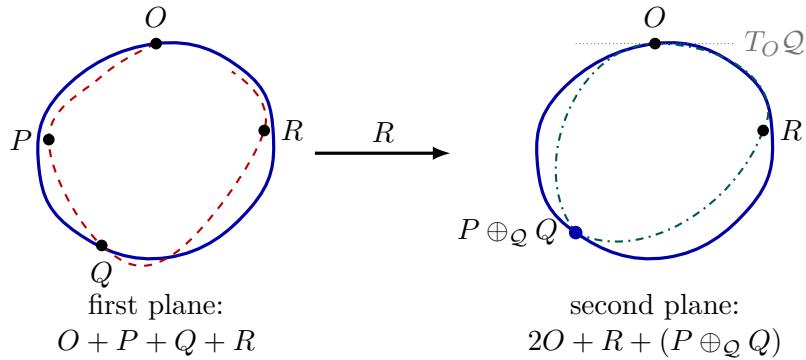

\end{document}